\documentclass[12pt]{article}
\usepackage{amsmath,enumerate,amsfonts,amssymb,color,graphicx,amsthm}

\usepackage{hyperref}
\usepackage{todonotes}
\usepackage[normalem]{ulem}
\usepackage{dsfont}
\usepackage[nolabel, final]{showlabels}

\theoremstyle{plain}
\newtheorem{lemma}{Lemma}[]
\newtheorem{thm}[lemma]{Theorem}
\newtheorem{prop}[lemma]{Proposition}

\newtheoremstyle{dotless}{}{}{\itshape}{}{\bfseries}{}{ }{}
\theoremstyle{dotless}
\newtheorem*{void*}{}
\theoremstyle{definition}

\theoremstyle{remark}
\newtheorem{remark}{Remark}

\newcommand{\M}{\mathcal{M}}
\newcommand{\R}{\mathbb{R}}

\newcommand{\di}{\mathrm{d}}
\newcommand{\disp}{\displaystyle}

\numberwithin{equation}{section} 

\title{Splitting theorems on complete Riemannian manifolds with nonnegative Ricci curvature}

\author{Alberto Farina\thanks{Universit\'e de Picardie Jules Verne. Email: alberto.farina@u-picardie.fr}, Jes\'us Oc\'ariz\thanks{Universidad Aut\'onoma de Madrid. Email: jesus.ocariz@uam.es}}

\begin{document}
\maketitle

\begin{abstract}
In this paper we provide some local and global splitting results on complete Riemannian manifolds with nonnegative Ricci curvature. We achieve the splitting through the analysis of some pointwise inequalities of Modica type which hold true for every bounded solution to a semilinear Poisson equation.
More precisely, we prove that the existence of a nonconstant bounded solution $u$ for which one of the previous inequalities becomes an equality at some point leads to the splitting results as well as to a classification of such a solution $u$.
\end{abstract}

\section{Introduction and Main Results}
Throughout this paper we shall denote by $(\M, g)$, or simply by $\M$, a complete, connected, smooth ($C^\infty$)  {boundaryless} Riemannian manifold of dimension $m\ge2$ with nonnegative Ricci curvature
and we shall consider the nonlinear Poisson equation over $\M$,

\begin{equation}\label{eqn:poisson}
-\Delta_g u + f(u)=0,
\end{equation}
where $f=F'$ is the first derivative of a function $F\in C^2(\mathbb{R})$ and $\Delta_g$ is the Laplace-Beltrami operator on $\M$. 

\vspace{0.5cm}

In the seminal paper \cite{modica1985gradient}, L. Modica proved the following pointwise gradient estimate for bounded solutions $u\in C^3(\R^m)$ of \eqref{eqn:poisson} with $F\geq 0 $ in the euclidean case:
\begin{equation}\label{eqn:modica}
\frac{1}{2} |\nabla u|^2(x) \leq F(u(x)) \hspace{1cm} \forall x\in \R^m.
\end{equation}

\vspace{0.1cm}

This important inequality, known as Modica's estimate, has been widely studied. In particular, it also holds in a Riemmanian geometric setting, i.e., the pointwise inequality 
	\begin{equation}\label{eqn:modicaM}
	\frac{1}{2} |\nabla_g u|^2(x) \leq F(u(x)) \hspace{1cm} \forall x\in \M
	\end{equation}
holds true under the assumptions that $(\M, g)$ is a complete Riemannian manifold with nonnegative Ricci curvature and $ F(t) \ge 0$ for every t $\in \R$, as proved in \cite{ratto1995gradient} and \cite{farina2010pointwise}.

\vspace{0.4cm}

Our first splitting theorem is

\begin{thm}\label{Thm:SplitClassicalModica}
Let $u\in C^3(\M)$ be a bounded solution of $\eqref{eqn:poisson}$ with $F\geq 0$. Suppose that there is a point $x_0$ such that equality is achieved at $\eqref{eqn:modicaM}$. Then, either $u$ is constant or $\M$ splits as the Riemannian product $\mathcal{N}\times \R$ where $\mathcal{N}\subset \M$ is a totally geodesic and isoparametric hypersurface with Ric($\mathcal{N})\geq 0$. Furthermore, $u:\mathcal{N}\times \R\to \mathbb{R}$ is such that $u(p,s)=\varphi(s)$ where $\varphi$ is a bounded and strictly monotone solution of the ODE, $\varphi^{''}=f(\varphi)$.
\end{thm}

\begin{remark}
An analogous conclusion about the one dimensionality of the solutions also appears in \cite{caffarelli1994gradient} for the euclidean case, i.e. $\R^m$ endowed with its standard flat metric. Nevertheless, the splitting result is totally new in the general geometric setting that we are considering here.

\end{remark}

The inequality \eqref{eqn:modicaM} is no longer true if we drop the assumption: $F(t) \geq 0$ for every $t \in \R$,
as one can easily see by considering the function $u(x) = \cos(x)$, which solves $ \Delta u = F'(u)$ in $\R^m$  with $F(t) = - \frac{t^2}{2}$. This problem is overcome in the next result where a refined and sharp form of Modica's estimate is obtained for any nonlinear function $ F \in C^2(\R).$

\begin{prop}\label{thm:ModGen} Let $u\in C^3(\M)$ be a bounded solution of $\eqref{eqn:poisson}$. Then,
\begin{equation}\label{eqn:modicaGen}
\frac{1}{2} |\nabla_g u|^2(x) \leq F(u(x)) -c_u \hspace{1cm} \forall x\in \M,
\end{equation}
where
\begin{equation}\label{eqn:c_u}
c_u:= \inf_{y\in \M} F(u(y)).
\end{equation}
\end{prop}

The above proposition recovers and improves the results in \cite{farina2008bernstein}, \cite{farina2010pointwise} established in the euclidean setting as well as those proved in \cite{farina2011pointwise} for compact Riemannian manifolds with nonnegative Ricci curvature. 

\begin{remark}\label{rem:sharpModGen}
Note that the bound \ref{eqn:modicaGen} is always sharp (by definition of $ c_u$) and that it is achieved everywhere by any entire 1D solution of a semilinear Poisson equation over $\R^m$ endowed with the flat metric (just multiply the ODE $u'' = F'(u)$ by $u'$ and then integrate the resulting identity). On the other hand, the bound \eqref{eqn:modicaM} is not necessarily sharp as shown by any 1D periodic solution of the Allen-Cahn equation $ \Delta u = u^3 - u$ on $\R^m$ (here $F(t) = \frac{(t^2-1)^2}{4}$). 

\end{remark}

\begin{remark}\label{rem:propM}
If $F\equiv 0$, we obtain, in particular, the classical result that the only bounded harmonic functions are the constant functions. It is well-known that the topological properties of connectivity, completeness and the absence of boundary are necessary so that Proposition \ref{thm:ModGen} is true. With respect to the curvature condition, on the one hand, if Ric$(\M)\geq 0$ with the previous properties implies the statement (Corollary 1 in \cite{yau1975harmonic}) and, on the other hand, if Ric$(\M)\leq -C$ (with $C>0$), there may exist nonconstant bounded harmonic functions (\cite{anderson1983dirichlet}).    
\end{remark}

The following proposition explains how the constant $c_u$ of the generalized Modica's estimate can be computed effectively using the extrema of the solution $u$.

\begin{prop}\label{prop:constantModica}
Let $u\in C^3(\M)$ be a bounded nonconstant solution of $\eqref{eqn:poisson}$. Then, $\eqref{eqn:modicaGen}$ holds with constant
$$
c_u=\min\left\{F\left(\inf_{x\in \M} u(x)\right),F\left(\sup_{x\in \M} u(x)\right)\right\}
$$
and
$$
c_u < F(t) \qquad \forall \,\, t \in \Big( \inf u, \, \sup u \Big) 
$$
\end{prop}

Besides, we study what happens when equality holds, at a regular point $x_0 \in \M,$ in the generalization of Modica's estimate \eqref{eqn:modicaGen}.
This leads to a local splitting of the manifold in a neighborhood of $x_0$ and to a precise description of the solution in this neighborhood. 

\begin{thm}\label{prop:Rigidity}
	Let $u$ be as in Proposition \ref{thm:ModGen}. Suppose that equality is achieved in $\eqref{eqn:modicaGen}$ at a regular point $x_0$, i.e. $\nabla_g u(x_0)\neq 0$. Then,
	\begin{itemize}
		\item equality in $\eqref{eqn:modicaGen}$ holds in the connected component of $\M \cap \{\nabla_g u \neq 0 \}$ that contains $x_0$,
		\item $Ric_g(\nabla_g u, \nabla_g u)$ vanishes at the connected component of $\M \cap \{\nabla_g u \neq 0 \}$
	    that contains $x_0$,
		\item there is a neighborhood of $x_0$, $\mathcal{U}\subset \M$, that splits as the Riemannian product $\mathcal{N}\times I$ where $\mathcal{N}\subset \M$ is a totally geodesic and isoparametric hypersurface with Ric($\mathcal{N})\geq 0$ and $I\subseteq \R$ is an interval,
		\item the solution $u$ restricted to the neighborhood $\mathcal{U}$, $u:\mathcal{N}\times I\to \mathbb{R}$, is equal to $u(p,s)=\varphi(s)$ where $\varphi$ is a bounded and strictly monotone solution of the ODE, $\varphi^{''}=f(\varphi)$.
	\end{itemize}
\end{thm}

\begin{remark}\label{rem:exampleLocalSplitting}
We underline that the result of local splitting obtained in the aforementioned Theorem \ref{prop:Rigidity} is sharp. Indeed, a global splitting result, like the one demonstrated in Theorem \ref{Thm:SplitClassicalModica}, is no longer true when $F$ is not nonnegative on $\R$. An example showing this phenomenon can be built in the following way : consider any Riemannian manifold $\mathcal{N}$ of dimension $m-1$, with nonnegative Ricci tensor and which does not contain any line (e.g. a round sphere) and set $ \M = \mathcal{N} \times \mathbb{S}^1$. The function $u(n,s) = \sin(s)$ is a solution of $\Delta_g u = -u$ for which the refined Modica's estimate \eqref{eqn:modicaGen} is achieved everywhere on $\M$ and in particular in many regular points of $u$. Nevertheless $\M$ does not split any euclidean factor.

\end{remark}

\begin{remark}\label{rem:exampleSingular} 
If a critical point achieves the equality in $\eqref{eqn:modicaGen}$, we cannot assure that there are more points which saturate the inequality, as shown by the following example in the euclidean space. 
For $m \geq3$ consider the function 
$u(x) = \Big ( \frac{{\sqrt{m(m-2)}}} {1 + \vert x \vert^2}  \Big )^{\frac{m-2}{2}}$, which is a radial smooth, positive, bounded solution of $ -\Delta u = u^{\frac{m+2}{m-2}}$ on the euclidean space $\R^m$.
In this case $ F(t) = - \frac{\vert t \vert^{p+1}}{p+1 \,\,}$, with $p = \frac{2m}{m-2},$ and $c_u = - \frac{(u(0))^{p+1}}{p+1}$ so, by integration we immediately get that $ \frac{1}{2} |\nabla u|^2 (x) = \frac{1}{2} \vert u' \vert^2 (\vert x\vert) = F(u(x)) - \frac{(u(0))^{p+1}}{p+1\,\,}$ for every $ x \in \R^m$. Therefore, the equality in 
$\eqref{eqn:modicaGen}$ is achieved at the origin, which is a critical point, while elsewhere inequality \eqref{eqn:modicaGen} is strict. 

\end{remark}

\vspace*{0.5cm}
The key point of the proof of the previous result is the construction of a local harmonic function whose gradient has constant length as long as the gradient of the solution does not vanish. This will be done in Section 4.

Some other splitting theorems were obtained in \cite{farina2013splitting} studying some kind of stable solutions of \eqref{eqn:poisson}.

The rest of the paper is devoted to the proofs of our contributions.

\vspace{0.5cm}
\textit{Acknowledgments}: 
The authors wish to thank Luciano Mari for useful discussions. The second author is partially supported by the grant MTM2017-83496-P from the Spanish Ministry of Economy and Competitiveness and through the “Severo Ochoa Programme for Centres of Excellence in R$\&$D”(SEV-2015-0554). The second author is also very grateful to the first author for his hospitality at Universit\'{e} de Picardie Jules Verne during a short research stay funded by the grant FPI with reference BES-2015-075151 from the Spanish government, in which this project was started.

\section{Proof of Propositions $\ref{thm:ModGen}$ and $\ref{prop:constantModica}$}

{\it {Proof of Proposition $\ref{thm:ModGen}$ : }} The result is already known for $\M = \R^m$ (see \cite{farina2008bernstein}, \cite{farina2010pointwise}) and when $\M$ is a compact manifold (see \cite{farina2011pointwise}). 

In the case of a complete, noncompact manifold we note that the gradient of the solution $u$ is bounded (see for instance Remark 48 in Appendix 1 of \cite{farina2013splitting} or Proposition 1 of \cite{ratto1995gradient}) and satisfies $ \inf_{\M} \vert \nabla u \vert =0$ (see for instance Appendix 1 of \cite{farina2013splitting}). The desired inequality \eqref{eqn:modicaGen} then follows from Theorem B of \cite{ratto1995gradient} applied with $ Q(u) = F(u) - c_u$. This concludes the proof of Proposition \ref*{thm:ModGen}. 

\bigskip

The key to prove Proposition \ref{prop:constantModica} is to see that if $c_u$ is achieved in an interior point of the interval $[\inf u,\sup u]$, then $u$ must be constant. This is the content of the next result. To prove it we follow \cite {farina2010pointwise} and  \cite{modica1985gradient}.

\medskip

\begin{prop}\label{prop:Liouville}
	Let $u\in C^3(\M)$ be a bounded solution of $\eqref{eqn:poisson}$ such that $c_u$ is achieved in an interior point of the interval $[\inf u,\sup u]$. Then, $u$ is constant.
	\end{prop}

\begin{proof} 
Suppose that $c_u$ is achieved at a interior point $u(x_0)=\alpha$, then $\alpha$ is a local minimum of $F$ and therefore verifies the following properties: $F(\alpha)=0$, $F'(\alpha)=0$ and $F''(\alpha)\geq 0$. Hence there exists $k\geq 0$ and $\delta >0$ such that
$$
F(s)-c_u\leq k (s-\alpha)^2
$$
for every $s$ such that $|s-\alpha|<\delta$.

Since $\M$ is connected, to obtain the desired conclusion it is enough to prove that the nonempty and closed set $A:=\left\{x\in \M: u(x)=\alpha \right\} $ is also open in $\M$. Pick $a \in A$. There is a normal neighborhood of $a$ of the form $B(a,r)$ for some $r=r(a) >0$ such that $|u(x)-\alpha| = |u(x)-u(a)| <\delta$ for every $x \in B(a, r)$ using the continuity of the function $u$. 
 
For any $x$ in $B(a, r) $ consider the unit speed minimizing geodesic $\gamma : [0, t_x] \rightarrow \M $ such that $ \gamma(0) = a$ and $ \gamma(t_x) = x$, where $t_x:=d(a,x)<r$. For all $ t \in [0,t_x]$ set 
$$
\varphi(t) := u(\gamma(t))-u(a)
$$
and observe that 
$$
\frac{1}{2}|\varphi'(t)|^2\leq \frac{1}{2} |\nabla_g u(\gamma(t))|^2\leq  F(u(\gamma(t))-c_u\leq k (u(\gamma(t))-\alpha)^2= k \varphi(t)^2
$$
since $ \gamma(t) \in B(a, r)$ for any $t \in [0,t_x]$ (recall that $ \gamma$ is minimizing). 

Therefore $ |\varphi'(t)|\leq C   \vert \varphi(t) \vert $ for some $C>0$ and for all $t \in [0,t_x]$ and so $ \varphi \equiv 0$ on $ [0,t_x],$ since $ \varphi(0) =0$. This means that $u(x) = u(a) = \alpha$ and so  $B(a, r) \subset A$, which concludes the proof.
\end{proof}

\begin{remark}\label{RemLiouville}
	Note that in the previous proof we have proved in particular the classical Liouville result. More specifically, if $c_u$ is a local minimum of $F$ and $u^{-1}(c_u)\neq \emptyset$, then $u$ must be constant. In particular, if a solution achieves the Modica's estimate \ref{eqn:modica} somewhere, it is either constant or doesn't have any critical points. Observe that all the hypothesis are necessary because otherwise there are easy counterexamples.
\end{remark}
\noindent {\it {Proof of Proposition $\ref{prop:constantModica}$.}} This is an immediate consequence of Proposition \ref{thm:ModGen} and the above Proposition \ref{prop:Liouville}.

\section{Some geometric results concerning the splitting}
To prove our splitting results we recall, for the sake of completeness, some results of Riemannian geometry (see for instance \cite{petersen2006riemannian}). 

\begin{lemma}
Let $E$ and $F$ be two distributions on $(\M,g)$, which are orthogonal complements of each other in $T\M$, and suppose that the distributions are parallel (i.e. if two vector fields $X$ and $Y$ are tangent to, say, $E$, then $\nabla_X Y$ is also tangent to $E$). Then,
\begin{itemize}
\item The distributions are integrable.
\item $\M$ is locally a product metric, i.e., there is a product neighborhood $U=V_E \times V_F$ such that $(U,g)=(V_E \times V_F, g_{|E}+g_{|F})$, where $g_{|E}$ and $g_{|F}$ are the restrictions of $g$ to the two distributions.
\end{itemize} 
\end{lemma}

\begin{proof}
To prove that the distributions are integrable we just need to show they are involutive because these concepts are equivalents thanks to Fr\"obenius Theorem. Thus, we want to show that having two vector fields $X$ and $Y$ tangent to, say, $E$, then the Lie bracket $[X,Y]$ is also tangent to $E$. Then,
$$
[X,Y]= \nabla_X Y - \nabla_Y X
$$
is tangent to $E$ because $E$ is parallel. Analogously with $F$.

\vspace{0.5cm}
Since $E$ and $F$ are integrable distributions, for any $p\in \M$ there is a neighborhood of $p$ where we can take local coordinates $x_i$ and $y_\alpha$ where $1\leq i\leq n$ and $n+1\leq \alpha \leq m$ such that
\begin{align*}
E&=\text{span} \left\{ \frac{\partial}{\partial x_i}: 1\leq i\leq n \right\}, \\
F&=\text{span} \left\{ \frac{\partial}{\partial y_\alpha}: n+1\leq \alpha \leq  m \right\}.
\end{align*}
We just need to show that $g_{|E}$ (respect. $g_{|F}$) is independent of $y_\alpha$ (respect. $x_i$), i.e. $\partial_\alpha g_{ij}=0$ (respect.$\partial_i g_{\alpha \beta}=0$)
$$
\partial_\alpha g_{ij} = \partial_\alpha g_{ij}+\partial_i g_{j\alpha}- \partial_j g_{i\alpha} = 2 \Gamma_{ij\alpha}= 2 g\left(\nabla_{\frac{\partial}{\partial x^j}}\frac{\partial}{\partial x^i},\frac{\partial}{\partial y^\alpha} \right) =0,
$$
where the previous equalities holds because $E$ and $F$ are orthogonal and $E$ is parallel. The equality $\partial_i g_{\alpha \beta}=0$ is proven analogously.
\end{proof}

\begin{lemma}
Let $X$ be a parallel vector field on $(\M,g)$. Then, 
\begin{itemize}
\item $X$ has constant length.
\item $X$ generates parallel distributions, one that contains $X$ and the other that is the orthogonal complement to $X$.
\item Locally the metric is a product with an interval, $(U,g)=(V\times I, g_{|TV}+ \mathrm{d}t^2)$.
\end{itemize} 
\end{lemma}
\begin{proof}
Since $X$ is parallel, we have by definition that $\nabla X= 0$. From the Fundamental Theorem of Riemmanian Geometry we get directly that $X$ has constant length:
$$
\nabla g(X,X)= g(\nabla X, X)+ g(X, \nabla X)= 0.
$$

\vspace{0.5cm}
We can consider the distribution $E$ and $F$ to be respectively the ones that for every $p\in \M$, $E_p=\text{span}\left\{X(p)\right\}$ and $F_p=\left( \text{span}\left\{ X(p)\right\}\right)^\bot$. These are parallel because if we have two vector fields $Y_1$ and $Y_2$ tangent to $E$, then clearly $\nabla_{Y_1} Y_2$ is tangent to $E$.

\vspace{0.5cm}

Using the previous lemma we get directly that the metric is locally a product with an interval.

\end{proof}

At last we recall the following result.

\begin{lemma}\label{parallel}
Let $u$ be a harmonic function on a manifold $\M$ with Ric$(\M)\geq 0$ such that $g(\nabla u, \nabla u)=1$. Then, $X=\nabla u$ is a parallel vector field.
\end{lemma}
\begin{proof}
Let $p\in \M$ and $\sigma$ be the integral curve of $\nabla u$ through $p$. Set $X=\nabla u$ and let $\{e_1,\cdots, e_{m-1}, X\}$ be an orthonormal frame in a neighborhood of $p$ which is parallel along $\sigma$. Then, $\nabla_X X=0$ because from $g(X,X)=0$ we get
$$
0=\nabla_X g(X,X)= 2 g(X, \nabla_X X),
$$
and for $e_i$ we get from the orthogonality and the condition that $e_i$ are parallel along $\sigma$,
$$
g(e_i, \nabla_X X)= \nabla_X g(e_i, X) - g(\nabla_X e_i, X)= \nabla_X 0- g(0, X)=0.
$$

At $p$, we have that
\begin{align*}
\text{Ric}(X) &= \sum_{i=1}^{m-1} g(R(e_i,X)X, e_i) + g(R(X,X)X, X)  \hspace{1.9cm} (R(X,X)X=0 \text{ from } \nabla_X X=0)\\
&=\sum_{i=1}^{m-1} g\left(\nabla_{e_i}\nabla_X X- \nabla_X \nabla_{e_i}X- \nabla_{[e_i,X]}X, e_i \right)  \hspace{1.8cm}(\text{using } \nabla_X X=0 \wedge \nabla_X e_i=0) \\
&=\sum_{i=1}^{m-1} \left[-g\left(\nabla_X \nabla_{e_i}X, e_i \right)- g\left(\nabla_{\nabla_{e_i} X} X, e_i \right)\right] \hspace{1.7cm}(\text{note } \nabla_{e_i} X=\sum_{j=1}^{m-1} g(\nabla_{e_i}X, e_j) e_j)\\
&= \sum_{i=1}^{m-1} \nabla_X \left[-g\left(\nabla_{e_i}X, e_i \right)\right]- \sum_{1\leq i,j\leq m-1} g(\nabla_{e_i}X, e_j)g(\nabla_{e_j}X, e_i) \hspace{0.8cm}(\text{using orthogonality}) \\
&= \nabla_X\left[-\text{div}(X)\right] - \sum_{1\leq k\leq m-1} g(\nabla_{e_k}X, e_k)^2 \hspace{2.5cm}(\text{div}(X)=0 \text{ from harmonicity})\\
&=- \Vert \nabla X \Vert ^2
\end{align*}
From the hypothesis that Ric$\geq 0$, we get that $\nabla X=0$, proving that $X$ is parallel.
\end{proof}
The previous geometric results are applied in the next section to prove Theorem 1 and Theorem 4. For that reason, a harmonic function, whose gradient has constant length, is constructed in a neighborhood of the point where the generalized Modica's estimate is achieved.

\section{Proof of Theorems $\ref{Thm:SplitClassicalModica}$ and $\ref{prop:Rigidity}$}

We are now ready to give the proof of Theorem \ref{prop:Rigidity}. 

\begin{proof}
Let us consider the function

\begin{equation}\label{eqn:defP}
P:= P(u,x)=\frac{1}{2}|\nabla_g u(x)|^2-F(u(x))+c_u,
\end{equation}
then by proceeding as in the proof of Theorem 1 of \cite{farina2011pointwise} we get 
\begin{equation}\label{eqn:PPMAX}
|\nabla_g u |^2 \Delta_g P -2f(u)\langle \nabla_g u, \nabla_g P \rangle-|\nabla_g P|^2 \geq |\nabla_g u |^2 Ric_g(\nabla_g u,\nabla_g u) \geq 0 \quad {\text {on}}
\quad \M
\end{equation}
Recall that $|\nabla_g u(x_0)| >0$ and  $P(u,x_0) = 0 $ by assumption. Also $ P \leq 0$ on $\M$ by \eqref{eqn:modicaGen} and therefore, in the light of \eqref{eqn:PPMAX}, the strong maximum principle 
gives that $P(u,x) = P(u,x_0) = 0$ 	in the  connected component of $\M\cap \{\nabla_g u\neq 0\}$ that contains $x_0$. The latter and \eqref{eqn:PPMAX} then imply $ Ric_g(\nabla_g u,\nabla_g u) = 0 $ in the  connected component of $\M\cap \{\nabla_g u\neq 0\}$ that contains $x_0$. Therefore we have proved the first two claims of Theorem \ref{prop:Rigidity}. 

Using the first statements of Theorem \ref{prop:Rigidity} we know that equality in \ref{eqn:modicaGen} holds in the connected component of $\M\cap \{\nabla_g u\neq 0\}$ where $x_0$ belongs, $\mathcal{U}$.

Now, from $u$ we are going to construct a harmonic function $v$ on $\mathcal{U}$ with constant gradient using a change of variables. We set $v:= H(u)$ where
$$
H(u):=\int_{u_0}^u \left(2F(s)-2c_u\right)^{-\frac{1}{2}} ds
$$
for some $u_0\in u(\mathcal{U})$. Then, $v$ has constant gradient in $\mathcal{U}$
$$
|\nabla_g v|^2 = H'(u)^2 |\nabla_g u|^2 = \frac{|\nabla_g u|^2}{2F(u)-2c_u} = 1,
$$
where we have used $\frac{1}{2}|\nabla_g u|^2=F(u)-c_u$ in the last equality. Let us check that $v$ is harmonic in $\mathcal{U}$
$$
\Delta_g v= H^{''}(u)|\nabla_g u|^2 + H'(u) \Delta_g u = -\frac{1}{2}\frac{f(u)}{(2F(u)-c_u)^{3/2}}|\nabla_g u|^2+\frac{f(u)}{(2F(u)-c_u)^{1/2}}=0
$$
using again that $\frac{1}{2}|\nabla_g u|^2=F(u)-c_u$ and $u$ being a solution of $\eqref{eqn:poisson}$.

Hence, since $v$ is a harmonic function whose gradient has length 1, we know that this generates a parallel vector field and we obtain a local splitting using the results from Section 3. 

Finally, the solution restricted to this neighborhood is 1D and monotone is straightforward.
\end{proof}

Now we turn to the proof of Theorem \ref{Thm:SplitClassicalModica}.

	\begin{proof}
		If $u$ is not constant, then by proceeding as in the first part of the proof of the previous result we see that $\vert \nabla v \vert = 1 $ on $\M$. Then,  the Bochner formula implies that the Hessian of $v$ is everywhere zero on $\M$ and so, the level sets of $v$ are totally geodesic (and isoparametric) smooth hypersurfaces (see Proposition 18 of \cite{farina2013splitting}). Moreover, from lemma \ref{parallel} we also know that $ \nabla v$ is a parallel vector field on $\M$. Therefore, if we denote by $\mathcal{N}$ the level hypesurface $\{v =0\}$ and by $ \phi_t$ the flow of $\nabla v$, it is well-known that the map $\Phi:  \mathcal{N} \times \R \to \M$ defined by 
		\[
		\Phi(x,t) := \phi_t(x)
		\]  
		is a Riemannian isometry  with respect to the product metric on $ \mathcal{N} \times \R$ (see e.g.  pp. 219-220 of \cite{Sakai} or pp. 206-207 of \cite{PRS}. See also Remark 2 on p. 145 of \cite{EH}). 
		
		Finally, in a local Darboux frame $\{e_j,\nu\ = \nabla v \}$ for the level surface $\mathcal{N}$, 
		\begin{equation}
		\begin{array}{lcl}
		0 = |II|^2 \Longrightarrow \nabla \di v (e_i,e_j) = 0, \\[0.2cm]
		0 = \langle \nabla |\nabla v|, e_j \rangle = \nabla \di v (\nu, e_j), 
		\end{array}
		\end{equation}
		so the unique nonzero component of $\nabla \di v$ is that corresponding to the pair $(\nu,\nu)$. Let $\gamma$ be any integral curve of $\nu$. Then
		$$
		\frac{\di}{\di t} (v \circ \gamma) = \langle \nabla v, \nu\rangle = |\nabla v| \circ \gamma > 0
		$$ 
		and so
		$$
		\begin{array}{lcl}
		\disp \frac{\di^2}{\di t^2} (v \circ \gamma)&=& \frac{\di}{\di t} (|\nabla v|\circ \gamma) = \langle \nabla |\nabla v|, \nu\rangle (\gamma) = \nabla \di v(\nu,\nu)(\gamma) \\[0.2cm]
		&=& \disp \Delta v (\gamma) = F'(v\circ \gamma).
		\end{array}
		$$
		Therefore $y = v\circ \gamma$ is a solution of the ODE $y'' = F'(y)$ and $y' >0$.

	\end{proof}

\bibliography{references}
\bibliographystyle{alpha}

\end{document}